\documentclass[11pt,reqno]{amsart}
\usepackage[T1]{fontenc}
\usepackage[utf8]{inputenc}
\usepackage[english]{babel}
\usepackage{lmodern}
\usepackage{amsmath,amssymb,amsthm,mathtools}
\usepackage[a4paper,margin=30mm]{geometry}
\usepackage{microtype}
\usepackage{enumitem}
\usepackage{hyperref}
\numberwithin{equation}{section}
\newtheorem{theorem}{Theorem}[section]
\newtheorem{proposition}[theorem]{Proposition}
\newtheorem{lemma}[theorem]{Lemma}
\newtheorem{corollary}[theorem]{Corollary}
\theoremstyle{definition}
\newtheorem{definition}[theorem]{Definition}
\theoremstyle{remark}
\newtheorem{remark}[theorem]{Remark}
\newcommand{\R}{\mathbb R}
\newcommand{\C}{\mathbb C}
\newcommand{\D}{\mathcal D}
\newcommand{\Id}{\operatorname{Id}}
\newcommand{\dd}{\mathrm d}
\newcommand{\Lie}{\mathcal L}
\newcommand{\Ric}{\operatorname{Ric}}
\newcommand{\tr}{\operatorname{tr}}
\newcommand{\Rea}{\operatorname{Re}}
\newcommand{\Ima}{\operatorname{Im}}
\newcommand{\SU}{\mathrm{SU}}
\newcommand{\Span}{\operatorname{span}}
\newcommand{\HS}{\mathrm{HS}}

\newcommand{\doi}[1]{\href{https://doi.org/#1}{\texttt{\nolinkurl{#1}}}}
\DeclareMathOperator{\rank}{rank}

\title[Quasi-contact metric structures from Sasaki-Einstein metrics]{Quasi-contact metric structures from Sasaki-Einstein metrics}
\author[B. Cappelletti-Montano]{Beniamino Cappelletti-Montano}
 \address{Dipartimento di Matematica e Informatica, Universit\`a degli Studi di Cagliari, Via Ospedale 72, 09124 Cagliari, Italy}
 \email{b.cappellettimontano@unica.it}

\author[A. De Nicola]{Antonio De Nicola}
 \address{Dipartimento di Matematica, Universit\`a degli Studi di Salerno, Via Giovanni Paolo II 132, 84084 Fisciano, Italy}
 \email{antondenicola@gmail.com}

\thanks{BCM was partially supported by the project ProBiki, fundend by Fondazione di Sardegna}

\date{}

\begin{document}
\begin{abstract}
We construct the first examples of quasi-contact metric structures that are not contact metric in every odd dimension greater than or equal to $5$. In dimension $5$, we construct a one-parameter family of such examples
with a Killing characteristic vector field by deforming
Sasaki--Einstein $\mathrm{SU}(2)$-structures. This yields compact, homogeneous examples on $S^5$. 
In arbitrary dimensions, we investigate quasi-contact metric structures with a Killing characteristic vector field and derive a  Ricci curvature criterion for the structure to be K-contact. 
\end{abstract}
\maketitle

\section{Introduction}

An almost contact metric manifold $(M^{2n+1},\varphi,\xi,\eta,g)$ is called
\emph{quasi-contact metric} if its associated almost Hermitian cone is
quasi-K\"ahler. More precisely, on $M\times\mathbb R$, with coordinate $t$,
one considers
\[
\widetilde g=e^{-2t}(g+\mathrm dt^2),\qquad
\widetilde JX=\varphi X-\eta(X)\partial_t,\qquad
\widetilde J\partial_t=\xi,
\]
and requires
\[
(\widetilde\nabla_U\widetilde J)V+
(\widetilde\nabla_{\widetilde JU}\widetilde J)\widetilde JV=0
\]
for all vector fields $U,V$ on the cone. Here $\widetilde\nabla$ is the
Levi-Civita connection of $\widetilde g$. The substitution $r=e^{-t}$
identifies this metric with the Riemannian cone metric
$\mathrm dr^2+r^2g$. The definition goes back to Tashiro's contact
$O^*$-structures \cite{Tashiro}, whereas its intrinsic formulation and its
relation to contact metric geometry were developed in \cite{KPS, PSS15}.

Every contact metric structure is quasi-contact metric since its cone is
almost K\"ahler, and every almost K\"ahler manifold is quasi-K\"ahler.
The converse implication has remained an unresolved question in the
literature. In particular, no quasi-contact metric example that is not
contact metric is provided in \cite{KPS,BJPS,MHO26}, and the existence of
such examples is explicitly described as open in the recent papers
\cite{MH24,MHO26}. 

We call a quasi-contact metric structure \emph{strict} if
it is not contact metric.  The absence of known strict examples leaves
open the possibility that the quasi-K\"ahler condition, although weaker
than the almost K\"ahler condition in general, might be too restrictive
on a cone to give a genuinely larger class. This naturally leads to the
question of whether every quasi-contact metric structure is necessarily
contact metric. In dimension three, the answer is affirmative: the cone
has complex dimension two, so the quasi-K\"ahler condition forces its
fundamental form to be closed.

The first purpose of this paper is to show that the answer to the above question is negative
in dimension $5$. A product construction then gives strict examples in
every higher odd dimension as well. We construct strict quasi-contact metric structures
on every $5$-manifold equipped with a Sasaki-Einstein
$\mathrm{SU}(2)$-structure. In particular, we obtain compact, complete
and simply connected examples of strict quasi-contact structures on $S^5$, whose   characteristic
vector fields are Killing. Thus, the distinction between quasi-contact
and contact metric structures persists even under this additional
metric condition. The $1$-form $\eta$ remains a contact form: this holds
for every quasi-contact metric structure \cite{PSS15}. What fails is its
compatibility condition $\mathrm d\eta=2 g(\cdot,\varphi\cdot)$ with the prescribed pair $(\varphi,g)$.

The construction is motivated by the description of $5$-dimensional
nearly Sasakian manifolds in terms of $\mathrm{SU}(2)$-structures due to
Cappelletti-Montano and Dileo \cite{CD}. Their work relates nearly
Sasakian and Sasaki–Einstein structures through deformations of the
underlying quaternionic structure in dimension $5$. Here we use the same quaternionic structure described in \cite{CD}, but we determine a different deformation by imposing the
quasi-K\"ahler condition on the cone. The essential operation is a
constant rotation of the horizontal endomorphisms, combined with
different constant changes of scale in the characteristic direction
and its orthogonal complement. 

Write the initial $SU(2)$-structure as
\[
(\eta_0,\omega_1,\omega_2,\omega_3;
 g_0,\xi_0,\varphi_1,\varphi_2,\varphi_3),\qquad
\omega_i(X,Y)=g_0(\varphi_iX,Y),
\]
where $\varphi_3$ is the Sasakian endomorphism, and the three
endomorphisms satisfy the quaternionic relations on $\ker\eta_0$ (see 
Section~\ref{sec:SE} for more details).

For an almost contact metric structure, we write
$h=\frac12\Lie_\xi\varphi$.
For an endomorphism $F$, the superscript $\ast$ denotes its
metric adjoint, defined by
\[
g(FX,Y)=g(X,F^\ast Y).
\]
For the structures below, $h_\mu=\frac12\Lie_{\xi_\mu}\varphi_\mu$,
and the adjoint is taken with respect to $g_\mu$.

Our main result is the following theorem.

\begin{theorem}\label{thm:main}
Let $M^5$ carry a Sasaki-Einstein $\mathrm{SU}(2)$-structure.
For $\mu\geq0$, put $\lambda=\sqrt{1+\mu^2}$ and define
\begin{align}
\eta_\mu&=\frac{3}{2\lambda}\eta_0,
&\xi_\mu&=\frac{2\lambda}{3}\xi_0,\label{eq:main-eta}\\
g_\mu&=\frac{3}{2\lambda^2}g_0+
\frac{3}{4\lambda^2}\eta_0\otimes\eta_0,
&\varphi_\mu&=\frac{\varphi_3+\mu\varphi_1}{\lambda}.
\label{eq:main-g}
\end{align}
Then $(\varphi_\mu,\xi_\mu,\eta_\mu,g_\mu)$ is quasi-contact metric
and $\xi_\mu$ is Killing. It is contact metric if and only if $\mu=0$;
at $\mu=0$ it is Sasakian. Moreover,
\begin{equation}\label{eq:main-h}
h_\mu=\mu\varphi_2,\qquad h_\mu^{\ast}=-h_\mu,\qquad
h_\mu^2=-\mu^2(\operatorname{Id}-\eta_\mu\otimes\xi_\mu),
\end{equation}
and its Levi-Civita connection satisfies
\begin{gather}\label{eq:main-derivative}
(\nabla_X\varphi_\mu)Y
=g_\mu(X+h_\mu X,Y)\xi_\mu
-\eta_\mu(Y)(X+h_\mu X),\\
R(X,\xi_\mu)\xi_\mu=(1+\mu^2)
\bigl\{X-\eta_\mu(X)\xi_\mu\bigr\}. \label{eq:main-Jacobi}
\end{gather}
If $g_0$ is complete, then $g_\mu$ is complete.
For $\mu>0$, the cone is quasi-K\"ahler but neither almost K\"ahler
nor nearly K\"ahler.
\end{theorem}

The global $\mathrm{SU}(2)$-structure exists, in particular, on every simply
connected Sasaki-Einstein $5$-manifold. For $S^5$, we construct all
the initial tensors explicitly from the standard complex volume form
on $\mathbb C^3$. This gives the following consequence.

\begin{corollary}\label{cor:sphere-intro}
The sphere $S^5$ admits strict quasi-contact metric structures with
Killing characteristic vector field. Their metrics are complete and
all their structure tensors are invariant under the transitive
$\mathrm{SU}(3)$-action. Dimension five is the smallest possible
dimension for a strict quasi-contact metric structure.
\end{corollary}

We verify the quasi-contact condition in two independent ways: by
computing the type of $\mathrm d\widetilde\Phi$ on the cone and by
calculating the covariant derivative of $\widetilde J$. We also
determine exactly which constant rotations and $2$-scale deformations
of a fixed Sasaki-Einstein structure are quasi-contact. This explains how the coefficients in
\eqref{eq:main-eta}--\eqref{eq:main-g} arise from the
quasi-contact condition.

The Killing property deserves particular attention because
Theorem~A of \cite{BJPS} asserts that it forces a quasi-contact metric
structure to be K-contact. The examples above disprove that assertion.
We identify the change-of-basis error in its proof, while retaining the
valid preliminary identities. In particular, on a quasi-contact metric
manifold, contact metric compatibility is equivalent to
$h^{\ast}=h$, whereas the Killing condition on $\xi$ is equivalent to
$h^{\ast}=-h$. These two requirements coincide only when $h=0$.
The following result gives a curvature criterion for that vanishing.

\begin{theorem}\label{thm:curvature-intro}
Let $(M^{2n+1},\varphi,\xi,\eta,g)$ be quasi-contact metric and suppose
that $\xi$ is Killing. Then
\begin{equation}\label{eq:intro-Jacobi}
R(X,\xi)\xi=X-\eta(X)\xi-h^2X,
\qquad
\operatorname{Ric}(\xi,\xi)=2n+\|h\|_{\mathrm{HS}}^2.
\end{equation}
Consequently, the structure is K-contact if and only if
$\operatorname{Ric}(\xi,\xi)=2n$ holds everywhere.
Every eigenspace of $-h^2|_{\ker\eta}$ corresponding to a positive
eigenvalue has real dimension divisible by $4$.
\end{theorem}

Here $\|h\|_{\mathrm{HS}}^2=\sum_j g(he_j,he_j)$ for an orthonormal
frame $(e_j)$ of $TM$. In particular, every plane containing $\xi$
has sectional curvature at least $1$. For the family in
Theorem~\ref{thm:main}, these sectional curvatures all equal
$1+\mu^2$.

The role of dimension five recalls a rigidity theorem in nearly
Sasakian geometry. Following the $5$-dimensional constructions in
\cite{CD}, De Nicola, Dileo and Yudin proved that every nearly Sasakian
manifold of dimension greater than $5$ is Sasakian \cite{DNDY}.
This parallel is useful, but it does not extend to 
quasi-contact geometry. Indeed, the product of two quasi-K\"ahler
cones is again quasi-K\"ahler. Taking the product of the cone over a
strict quasi-contact manifold $M^5$  with the flat K\"ahler cone over the unit
sphere $S^{2q-1}$, $q\geq1$, and writing
$r_1=\rho\cos\theta$, $r_2=\rho\sin\theta$, gives
\[
\mathrm d\rho^2+\rho^2
\bigl(\mathrm d\theta^2+\cos^2\theta\,g_\mu+
\sin^2\theta\,g_{S^{2q-1}}\bigr).
\]
The product almost complex structure is invariant under radial
dilations and therefore induces an almost contact metric structure
on 
$M^5\times S^{2q-1}\times(0,\pi/2)$.
This structure is quasi-contact and remains strict since the
exterior derivative of the product fundamental form is the nonzero
pullback of that of the first factor. Hence, strict examples also
exist in every odd dimension greater than $5$. Their characteristic
fields are again Killing, but their characteristic Jacobi operators have
two distinct eigenvalues. Section~\ref{sec:joins} gives all the induced
tensors and develops the construction in detail. 

The resulting manifolds are geodesically incomplete, since the
angular geodesics reach the excluded endpoints of the
$\theta$-interval in finite time.
The analogy with nearly Sasakian geometry, nevertheless, suggests a
question in the complete setting:
\begin{quote}
\emph{Is every complete quasi-contact metric manifold of dimension
greater than $5$ necessarily contact metric?}
\end{quote}
The above construction does not answer this question.
Indeed, by Proposition~\ref{prop:join-blowup}, the sectional curvature
becomes unbounded at a finite-distance end, excluding the possibility of any smooth
Riemannian extension across that end.
In dimensions greater than or equal to $5$, there is a partial result 
in \cite{MH24} under the additional assumption that the Riemannian curvature tensor satisfies a nullity condition.


\section{Cone conventions and the quasi-contact condition}\label{sec:cones}


An almost contact metric structure on $M^{2n+1}$ consists of  $(\varphi,\xi,\eta,g)$ satisfying
\begin{gather}
\varphi^2=-\Id+\eta\otimes\xi,\qquad
\eta(\xi)=1,\qquad \varphi\xi=0,\qquad \eta\circ\varphi=0,\label{eq:ac-algebra}\\
g(\varphi X,\varphi Y)=g(X,Y)-\eta(X)\eta(Y),
\qquad g(\xi,X)=\eta(X).
\label{eq:ac-metric}
\end{gather}
We put
$\D=\ker\eta$, $P=\Id-\eta\otimes\xi$ and
$\Phi(X,Y)=g(X,\varphi Y)$.
Thus $P$ is the orthogonal projection onto $\D$.

For an endomorphism $F$ of $TM$, the symbol $F^{\ast}$ denotes its pointwise
adjoint with respect to the metric under consideration:
\[
g(FX,Y)=g(X,F^{\ast}Y).
\]
On $\widetilde M=M\times\R$, with coordinate $t$ and $T=\partial_t$, define
\begin{equation}\label{eq:cone}
\widetilde g=e^{-2t}(g+\dd t^2),\qquad
\widetilde JX=\varphi X-\eta(X)T,\qquad
\widetilde JT=\xi.
\end{equation}
With $r=e^{-t}$, this is the Riemannian cone metric $\dd r^2+r^2g$ on $M\times\R_{>0}$. Its fundamental $2$-form is
\begin{equation}\label{eq:cone-F}
\widetilde\Phi=e^{-2t}(\Phi+\eta\wedge\dd t),
\end{equation}
so that
\begin{equation}\label{eq:cone-dF}
\dd\widetilde\Phi=e^{-2t}
\{\dd\Phi+\dd t\wedge(\dd\eta-2\Phi)\}.
\end{equation}
It follows that $\dd\widetilde\Phi=0$ if and only if $\dd\eta=2\Phi$: this equation also implies $\dd\Phi=0$ by applying $\dd$. Throughout this paper, we use the Spivak/Lee convention for the exterior derivative $\mathrm d$ and the wedge product. 

\begin{definition}
We call $(\varphi,\xi,\eta,g)$ \emph{contact metric} if $\dd\eta=2\Phi$, and \emph{quasi-contact metric} if the cone \eqref{eq:cone} is quasi-K\"ahler, that is,
\begin{equation}\label{eq:QK}
(\widetilde\nabla_U\widetilde J)V+
(\widetilde\nabla_{\widetilde JU}\widetilde J)\widetilde JV=0
\end{equation}
for all vector fields $U,V$ on $\widetilde M$.
A contact metric structure is \emph{K-contact} if $\xi$ is Killing, and \emph{Sasakian} if
\begin{equation}\label{eq:Sasakian}
(\nabla_X\varphi)Y=g(X,Y)\xi-\eta(Y)X.
\end{equation}
\end{definition}

In \cite{KPS,BJPS}, the contact metric equation is written $d\eta=\Phi$, with the differential convention
$d\eta(X,Y)=\tfrac12\left(X(\eta(Y))-Y(\eta(X))-\eta([X,Y])\right)$.
Our $\dd\eta$ is twice this $2$-form. We use the usual exterior derivative uniformly in every degree.

\medskip
We use the type decomposition of complexified forms on an almost complex manifold without assuming integrability. A complex $1$-form $\zeta$ has type $(1,0)$ when $\zeta(JU)=i\zeta(U)$.

\begin{lemma}\label{lem:QK-type}
Let $(N,J,G)$ be almost Hermitian, with $F(U,V)=G(U,JV)$. Then
\[
(\nabla_UJ)V+(\nabla_{JU}J)JV=0
\]
if and only if $\dd F$ has type $(3,0)+(0,3)$.
\end{lemma}

\begin{proof}
Set $\mathcal A(U,V,W)=(\nabla_UF)(V,W)$ and extend this tensor complex linearly. It is skew-symmetric in its last two arguments. Differentiating $J^2=-\Id$ gives
\[
\mathcal A(U,JV,JW)=-\mathcal A(U,V,W).
\]
Therefore $\mathcal A$ vanishes whenever its last two arguments have opposite types. The quasi-K\"ahler equation is equivalent to
\[
\mathcal A(U,V,W)+\mathcal A(JU,JV,W)=0.
\]
After complexification, this equation eliminates the components in which the first two arguments have opposite types. Consequently, only the components with all three arguments of the same type can remain, and their cyclic alternation $\dd F$ has type $(3,0)+(0,3)$.

Conversely, take $Z_1,Z_2\in T^{1,0}N$ and $\overline Z_3\in T^{0,1}N$. Since the mixed components in the last two arguments vanish,
\[
\dd F(\overline Z_3,Z_1,Z_2)
=\mathcal A(\overline Z_3,Z_1,Z_2).
\]
The absence of a $(2,1)$-component of $\dd F$, together with its conjugate statement, therefore eliminates every component of $\mathcal A$ except those with all arguments of the same type. The quasi-K\"ahler equation follows.
\end{proof}

This is the usual characterization of the quasi-K\"ahler class in almost Hermitian geometry; see \cite{GH}. We have included the proof because it makes the cone verification independent of an intrinsic characterization of quasi-contact structures.

\begin{corollary}\label{cor:dim3}
Every $3$-dimensional quasi-contact metric structure is contact metric.
\end{corollary}
\begin{proof}
The cone has complex dimension $2$, so there are no nonzero forms of type $(3,0)$ or $(0,3)$. Lemma~\ref{lem:QK-type} gives $\dd\widetilde\Phi=0$, and \eqref{eq:cone-dF} yields $\dd\eta=2\Phi$.
\end{proof}

\medskip
Let $B$ be a horizontal $2$-form, i.e. a $2$-form such that $i_{\xi}B=0$.
We say that $B$ is $\varphi$-anti-invariant if
$B(\varphi X,\varphi Y)=-B(X,Y)$, for any $X,Y\in\D$.
In that case $B_\varphi(X,Y):=B(\varphi X,Y)$ is again a $2$-form, and $B-iB_\varphi$ has horizontal type $(2,0)$.

\begin{proposition}\label{prop:horizontal-test}
An almost contact metric structure is quasi-contact if and only if the following conditions hold, with $B=\dd\eta-2\Phi$:
\begin{align}
\iota_\xi\dd\eta&=0,\label{eq:hor1}\\
B(\varphi X,\varphi Y)&=-B(X,Y),\label{eq:hor2}\\
\dd\Phi(\xi,X,Y)&=B(\varphi X,Y),\label{eq:hor3}
\end{align}
for $X,Y\in\D$, and $\dd\Phi|_{\D}$ has horizontal type $(3,0)+(0,3)$. In dimension $5$, the last condition is simply $\dd\Phi|_{\D}=0$.
\end{proposition}

\begin{proof}
The $(1,0)$-cotangent space of the cone splits as the horizontal $(1,0)$-space and the complex line generated by $\zeta=\dd t+i\eta$, which has type $(1,0)$. Write
\[
B=\eta\wedge\alpha+B_0,\qquad
\dd\Phi=\eta\wedge\sigma+\tau,
\]
where $\alpha,B_0,\sigma,\tau$ are horizontal. By \eqref{eq:cone-dF}, up to the positive factor $e^{-2t}$ the exterior derivative of the cone form is
\[
\tau+\eta\wedge\sigma+\dd t\wedge B_0
+\dd t\wedge\eta\wedge\alpha.
\]
A form of pure type $(3,0)+(0,3)$ cannot contain two factors from the real plane generated by $\dd t$ and $\eta$, so $\alpha=0$. The component with no such factor must be horizontal of pure type. The remaining part must be the real part of $\zeta\wedge\psi$ for a horizontal $(2,0)$-form $\psi$. Writing
\[
\psi=B_0-i(B_0)_\varphi
\]
shows that this is equivalent to anti-invariance of $B_0$ and $\sigma=(B_0)_\varphi$. These are exactly \eqref{eq:hor1}--\eqref{eq:hor3}. In dimension five, the horizontal complex dimension is $2$.
\end{proof}

\begin{corollary}[\cite{PSS15}]\label{cor:eta-contact}
The $1$-form of any quasi-contact metric structure is a contact form, and its characteristic field is its Reeb vector field.
\end{corollary}
\begin{proof}
For $X\in\D$, anti-invariance and skew-symmetry imply $B(X,\varphi X)=0$. Hence
\[
\dd\eta(X,\varphi X)=2\Phi(X,\varphi X)=-2g(X,X).
\]
No nonzero horizontal vector can belong to the kernel of $\dd\eta|_{\D}$. Thus this $2$-form is nondegenerate and $\eta\wedge(\dd\eta)^n$ never vanishes. Equation \eqref{eq:hor1}, together with $\eta(\xi)=1$, identifies the Reeb field.
\end{proof}

\medskip
For vector fields on $M$, the Levi-Civita connection of the metric \eqref{eq:cone} is
\begin{equation}\label{eq:cone-LC}
\begin{aligned}
\widetilde\nabla_XY&=\nabla_XY+g(X,Y)T,&
\widetilde\nabla_XT&=-X,\\
\widetilde\nabla_TX&=-X,&
\widetilde\nabla_TT&=-T.
\end{aligned}
\end{equation}
Consequently,
\begin{align}
(\widetilde\nabla_X\widetilde J)Y
={}&(\nabla_X\varphi)Y+\eta(Y)X-g(X,Y)\xi
+\{\Phi(X,Y)-(\nabla_X\eta)(Y)\}T,
\label{eq:cone-nabla1}\\
(\widetilde\nabla_X\widetilde J)T
={}&\nabla_X\xi+\varphi X,
\qquad \widetilde\nabla_T\widetilde J=0.
\label{eq:cone-nabla2}
\end{align}

The following identities are also recorded in \cite{KPS,BJPS}. We derive them directly from the cone.

\begin{lemma}\label{lem:h-identities}
On a quasi-contact metric manifold, put $h=\frac12\Lie_\xi\varphi$. Then
\begin{gather}
\nabla_\xi\xi=0,\qquad \nabla_\xi\varphi=0,\qquad
h\xi=0,\qquad \eta\circ h=0,\qquad h\varphi=-\varphi h,\label{eq:h-basic}\\
\nabla_X\xi=-\varphi X-\varphi hX.
\label{eq:h-nabla-xi}
\end{gather}
If $h_a=\frac12(h-h^*)$ is the skew-adjoint part of $h$, then
\begin{equation}\label{eq:contact-defect}
\dd\eta(X,Y)-2\Phi(X,Y)=-2g(\varphi h_aX,Y).
\end{equation}
In particular, the structure is contact metric if and only if $h=h^*$.
\end{lemma}

\begin{proof}
Use \eqref{eq:QK} with $U=T$. Since $\widetilde\nabla_T\widetilde J=0$ and $\widetilde JT=\xi$, one obtains $\widetilde\nabla_\xi\widetilde J=0$. Equations \eqref{eq:cone-nabla1}--\eqref{eq:cone-nabla2} then give $\nabla_\xi\xi=0$ and $\nabla_\xi\varphi=0$.

Write $S(X)=\nabla_X\xi$. The vector field $\xi$ has unit length, so $S$ takes values in $\D$, and $S\xi=0$. Taking $U=X$, $V=T$ in \eqref{eq:QK} gives
\[
S-\varphi S\varphi=-2\varphi.
\]
On the other hand, the torsion-free property of $\nabla$ yields
\[
2h=\Lie_\xi\varphi=\nabla_\xi\varphi-S\varphi+\varphi S
=-S\varphi+\varphi S.
\]
Multiplication by $\varphi$ and use of the preceding identity give $\varphi h=-S-\varphi$, which is \eqref{eq:h-nabla-xi}.

Since $\nabla_\xi\xi=0$, the identity $\Lie_\xi\eta=\iota_\xi\dd\eta=0$ follows directly from unit length. Lie differentiation of \eqref{eq:ac-algebra} now gives $h\xi=0$, $\eta\circ h=0$ and $h\varphi+\varphi h=0$. The adjoint $h^*$ also anticommutes with $\varphi$. Therefore
\[
S-S^*=-2\varphi-\varphi(h-h^*)=-2\varphi-2\varphi h_a.
\]
Pairing this identity with $X,Y$ gives \eqref{eq:contact-defect}. The last assertion follows because $\varphi$ is invertible on $\D$ and $h_a$ vanishes on $\xi$ and takes values in $\D$.
\end{proof}

\begin{proposition}\label{prop:intrinsic-identity}
Every quasi-contact metric structure satisfies
\begin{equation}\label{eq:general-intrinsic-QC}
\begin{split}
(\nabla_X\varphi)Y+(\nabla_{\varphi X}\varphi)\varphi Y
={}&2g(X,Y)\xi-\eta(Y)X-\eta(X)\eta(Y)\xi-\eta(Y)hX.
\end{split}
\end{equation}
\end{proposition}
\begin{proof}
Put $\widetilde JX=\varphi X-\eta(X)T$ and
$\widetilde JY=\varphi Y-\eta(Y)T$ into the quasi-K\"ahler equation.
By \eqref{eq:cone-nabla1}--\eqref{eq:cone-nabla2}, the tangent part of its left-hand side is
\begin{align*}
 &(\nabla_X\varphi)Y+\eta(Y)X-g(X,Y)\xi
 +(\nabla_{\varphi X}\varphi)\varphi Y
 -g(\varphi X,\varphi Y)\xi-\eta(Y)\{\nabla_{\varphi X}\xi+\varphi^2X\}.
\end{align*}
Equation \eqref{eq:h-nabla-xi} and $h\varphi=-\varphi h$ give
$\nabla_{\varphi X}\xi+\varphi^2X=-hX$.
Using $g(\varphi X,\varphi Y)=g(X,Y)-\eta(X)\eta(Y)$ yields
\eqref{eq:general-intrinsic-QC}. This is the identity in
\cite[Theorem~2.1]{BJPS}, obtained here directly from the cone.
\end{proof}

\begin{lemma}\label{lem:NK-cone}
A cone of the form \eqref{eq:cone} is nearly K\"ahler if and only if the underlying almost contact metric structure is Sasakian. In this case the cone is K\"ahler.
\end{lemma}
\begin{proof}
On a nearly K\"ahler manifold, $\widetilde\nabla\widetilde\Phi$ is a $3$-form and $\dd\widetilde\Phi=3\widetilde\nabla\widetilde\Phi$. By \eqref{eq:cone-nabla2}, its covariant derivative in the $T$-direction vanishes. Thus $\iota_T\dd\widetilde\Phi=0$, and \eqref{eq:cone-dF} gives $\dd\eta=2\Phi$. It follows that $\dd\widetilde\Phi=0$, hence $\widetilde\nabla\widetilde J=0$. Equations \eqref{eq:cone-nabla1}--\eqref{eq:cone-nabla2} imply \eqref{eq:Sasakian}. The converse follows by substitution in those same equations.
\end{proof}

\section{Sasaki-Einstein \texorpdfstring{$\SU(2)$}{SU(2)}-structures in dimension 5}\label{sec:SE}


An $\SU(2)$-structure on a $5$-manifold can be described by three almost contact metric structures $(\varphi_i,\xi_0,\eta_0,g_0)$ sharing the same metric, characteristic field and $1$-form. On $\D=\ker\eta_0$ their endomorphisms satisfy the quaternionic relations
\begin{equation}\label{eq:quaternion}
\varphi_\alpha\varphi_\beta=\varphi_\gamma,
\end{equation}
for a cyclic permutation $(\alpha,\beta,\gamma)$ of $\left\{1,2,3\right\}$. Define
$\omega_i(X,Y)=g_0(\varphi_iX,Y)$ (notice the order of the arguments, the fundamental form of $(\varphi_i,\xi_0,\eta_0,g_0)$ being $-\omega_i$).
We shall consider Sasaki-Einstein $\SU(2)$-structures, which are
characterized by the following additional differential equations:
\begin{equation}\label{eq:SE-d}
\dd\eta_0=-2\omega_3,\qquad
\dd\omega_1=3\eta_0\wedge\omega_2,\qquad
\dd\omega_2=-3\eta_0\wedge\omega_1,\qquad
\dd\omega_3=0.
\end{equation}
These are the Sasaki-Einstein equations in \cite[Section~4]{CD}.
The associated metric satisfies $\Ric_{g_0}=4g_0$, and
$(\varphi_3,\xi_0,\eta_0,g_0)$ is Sasakian.
 We will need the following covariant form of the equations:
\begin{align}
\nabla^0_X\xi_0&=-\varphi_3X,\label{eq:SE-xi}\\
(\nabla^0_X\varphi_3)Y
&=g_0(X,Y)\xi_0-\eta_0(Y)X,\label{eq:SE-phi3}\\
(\nabla^0_X\varphi_1)Y
&=\eta_0(X)\varphi_2Y-\eta_0(Y)\varphi_2X
+\omega_2(X,Y)\xi_0,\label{eq:SE-phi1}\\
(\nabla^0_X\varphi_2)Y
&=-\eta_0(X)\varphi_1Y+\eta_0(Y)\varphi_1X
-\omega_1(X,Y)\xi_0.
\label{eq:SE-phi2}
\end{align}
For example, \eqref{eq:SE-phi1}--\eqref{eq:SE-phi2} also follow from the nearly cosymplectic structures described in \cite[Corollary~5.2]{CD}. 


We now explain how the supplementary tensors of the
$\mathrm{SU}(2)$-structure arise from the K\"ahler cone.
Let
$(M^5,\varphi_3,\xi_0,\eta_0,g_0)$
be a Sasaki-Einstein manifold, where
$(\varphi_3,\xi_0,\eta_0,g_0)$ denotes its given Sasakian structure.
At this stage, only the Sasakian endomorphism $\varphi_3$ is being
used; the additional endomorphisms $\varphi_1$ and $\varphi_2$ will
be obtained below from a parallel complex volume form on the cone.

Consider the Riemannian cone
\[
C(M)=M\times\mathbb R_{>0},\qquad
G_0=\dd r^2+r^2g_0.
\]
Since $M$ is Sasaki-Einstein, $G_0$ is Ricci-flat K\"ahler.
We choose the sign of the cone complex structure $I$ so that
\[
I(r\partial_r)=\xi_0.
\]
For every vector field $X$ tangent to $M$, one then has
\begin{equation}\label{eq:cone-I}
IX=-\varphi_3X-\eta_0(X)\,r\partial_r.
\end{equation}
Let us compare this convention with the cone almost complex structure
$\widetilde J$ introduced in Section~\ref{sec:cones}. There we used
the coordinate $t$, with $r=e^{-t}$, and
\[
\widetilde JX=\varphi_3X-\eta_0(X)\partial_t,
\qquad
\widetilde J\partial_t=\xi_0.
\]
Since
$\partial_t=-r\partial_r$,
it follows that
\[
\widetilde J(r\partial_r)=-\xi_0,
\qquad
\widetilde JX=\varphi_3X+\eta_0(X)r\partial_r.
\]
Hence the K\"ahler complex structure used in the present section is
simply
\[
I=-\widetilde J.
\]
Of course, replacing an almost complex structure by its negative does
not affect the K\"ahler condition.
Restrict now to the hypersurface $r=1$ and put
\[
N=\partial_r.
\]
Then \eqref{eq:cone-I} becomes
\[
IX=-\varphi_3X-\eta_0(X)N.
\]
Suppose that the cone carries a nowhere vanishing parallel
$(3,0)$-form $\Omega$. We fix its constant scale by requiring
\[
\frac{i}{8}\Omega\wedge\overline{\Omega}
=\operatorname{vol}_{G_0},
\]
where the volume form is taken with the complex orientation
induced by $I$.
Define
\begin{equation}\label{eq:alpha-cone}
\alpha(X,Y)=\Omega(N,X,Y),\qquad
\omega_1=\Rea\alpha,\qquad
\omega_2=-\Ima\alpha.
\end{equation}
The form $\alpha$ is horizontal because $\xi_0=IN$ and $\Omega$ is complex multilinear. In a unitary frame the real and imaginary parts of \eqref{eq:alpha-cone}, together with $\omega_3=g_0(\varphi_3\cdot,\cdot)$, give \eqref{eq:quaternion}.
The cone connection satisfies $\overline\nabla_XN=X$ along $r=1$ and
\[
\overline\nabla_XY=\nabla^0_XY-g_0(X,Y)N.
\]
Since $\Omega$ is parallel, terms containing two copies of $N$ vanish and
\begin{equation}\label{eq:alpha-derivative}
(\nabla^0_X\alpha)(Y,Z)=\Omega(X,Y,Z).
\end{equation}
On three horizontal vectors, $\Omega$ vanishes because $\D$ has complex dimension $2$. Moreover,
$\Omega(\xi_0,Y,Z)=i\alpha(Y,Z)$. Therefore
\begin{equation}\label{eq:alpha-tangent}
\Omega|_{TM}=i\eta_0\wedge\alpha,
\qquad
(\nabla^0_X\alpha)(Y,Z)=i(\eta_0\wedge\alpha)(X,Y,Z).
\end{equation}
Taking real and imaginary parts yields \eqref{eq:SE-phi1}--\eqref{eq:SE-phi2}. Alternating gives
$\dd\alpha=3i\eta_0\wedge\alpha$, which is the middle pair of equations in \eqref{eq:SE-d}. Finally, $\overline\nabla I=0$ gives \eqref{eq:SE-xi}--\eqref{eq:SE-phi3}, and hence the remaining exterior equations.

\begin{proposition}\label{prop:global-SE}
Let $(M^5,\varphi_3,\xi_0,\eta_0,g_0)$ be a simply connected
Sasaki-Einstein manifold. Then there exist globally defined
$2$-forms $\omega_1,\omega_2$ on $M$ such that, setting
$\omega_3(X,Y)=g_0(\varphi_3X,Y)$,
the quadruple
\[
(\eta_0,\omega_1,\omega_2,\omega_3)
\]
defines a Sasaki-Einstein $\SU(2)$-structure satisfying
\eqref{eq:SE-d}.
\end{proposition}
\begin{proof}
Let $(C(M),I,G_0)$ be the K\"ahler cone of the given
Sasaki-Einstein structure. Since $M$ is simply connected and
$\mathbb R_{>0}$ is contractible, $C(M)=M\times\mathbb R_{>0}$
is simply connected. Moreover, $G_0$ is Ricci-flat.
The connection induced on the canonical line bundle
\[
K_{C(M)}=\Lambda^{3,0}T^*C(M)
\]
has curvature proportional to the Ricci form and is therefore flat.
Since the cone is simply connected, this connection has trivial
holonomy. Consequently, there exists a nowhere vanishing global
parallel $(3,0)$-form $\Omega$. Multiplying $\Omega$ by a positive
constant, we choose the same algebraic normalization as in the
preceding construction.

Identify $M$ with the hypersurface $r=1$ and put $N=\partial_r$.
The construction in \eqref{eq:alpha-cone} now gives globally
defined forms
\[
\alpha(X,Y)=\Omega(N,X,Y),\qquad
\omega_1=\Rea\alpha,\qquad
\omega_2=-\Ima\alpha.
\]
Define the global endomorphisms $\varphi_1,\varphi_2$ by
\[
g_0(\varphi_iX,Y)=\omega_i(X,Y),\qquad i=1,2.
\]
The algebraic verification above shows that these endomorphisms,
together with the given $\varphi_3$, satisfy
\eqref{eq:quaternion} and define almost contact metric structures
with the same $\xi_0$, $\eta_0$ and $g_0$.

Finally, the identities \eqref{eq:alpha-derivative} and
\eqref{eq:alpha-tangent} apply to this global parallel form and
yield the differential equations \eqref{eq:SE-d}, as established
above. Hence
\[
(\eta_0,\omega_1,\omega_2,\omega_3;
g_0,\xi_0,\varphi_1,\varphi_2,\varphi_3)
\]
is a globally defined Sasaki-Einstein $\SU(2)$-structure
extending the original Sasakian structure.
\end{proof}

\begin{remark}
On a nonsimply connected Sasaki-Einstein manifold the canonical bundle of the cone can have nontrivial flat holonomy. Our construction requires a global Sasaki-Einstein $\SU(2)$-structure, not merely locally defined forms $\omega_1,\omega_2$. It applies without this global assumption on each open set where such forms are defined.
\end{remark}

Note that Cartan's formula and \eqref{eq:SE-d} give
\begin{equation}\label{eq:SE-Lie}
\Lie_{\xi_0}\omega_1=3\omega_2,\qquad
\Lie_{\xi_0}\omega_2=-3\omega_1,\qquad
\Lie_{\xi_0}\omega_3=0.
\end{equation}
Since $\xi_0$ is Killing, these identities also hold with $\omega_i$ replaced by $\varphi_i$. They describe the rotation of the supplementary endomorphisms along the characteristic flow.

\section{Constant rotations and 2-scale deformations}\label{sec:deformations}


Fix a Sasaki-Einstein $\SU(2)$-structure as in
Section~\ref{sec:SE}. Let $a,b>0$ and $c,s\in\R$ be constants satisfying $c^2+s^2=1$. Define
\begin{equation}\label{eq:ansatz}
\begin{aligned}
\eta&=b\eta_0,\qquad \xi=b^{-1}\xi_0,\\
g&=a g_0+(b^2-a)\eta_0\otimes\eta_0,
\qquad \varphi=c\varphi_3+s\varphi_1.
\end{aligned}
\end{equation}
The metric is $a g_0$ on $\D$ and gives $\xi_0$ squared length $b^2$. The quaternionic relations show that $\varphi^2=-\Id+\eta\otimes\xi$ and that $\varphi$ is orthogonal on $\D$. Thus \eqref{eq:ansatz} is almost contact metric. Its characteristic vector field $\xi$ is Killing because the initial characteristic flow preserves both $g_0$ and $\eta_0$.
Set $\widehat\omega_i=a\omega_i$ and define
\begin{equation}\label{eq:beta-gamma}
\chi=s\widehat\omega_1+c\widehat\omega_3=-\Phi,
\quad
\beta=c\widehat\omega_1-s\widehat\omega_3,
\quad
\gamma=\widehat\omega_2.
\end{equation}
The endomorphisms corresponding to $\beta$ and $\gamma$ are
\[
F_\beta=c\varphi_1-s\varphi_3,\qquad
F_\gamma=\varphi_2.
\]
Using \eqref{eq:quaternion}, we obtain
\begin{equation}\label{eq:PQ}
F_\beta\varphi=-F_\gamma,\qquad
F_\gamma\varphi=F_\beta.
\end{equation}
Therefore
\begin{equation}\label{eq:beta-types}
\beta(\varphi X,Y)=-\gamma(X,Y),\qquad
\gamma(\varphi X,Y)=\beta(X,Y),
\end{equation}
and $\beta+i\gamma$ has horizontal type $(2,0)$. The form $\chi$ is of type $(1,1)$.

\begin{theorem}\label{thm:ansatz}
The structure \eqref{eq:ansatz} is quasi-contact if and only if one of the following two alternative conditions holds:
\begin{enumerate}[label=\textup{(\roman*)},leftmargin=2.6em]
\item $s=0$, $c=1$ and $a=b$;
\item $s\neq0$, $a=\frac23b^2$ and $c=\frac23b$.
\end{enumerate}
In the first case the structure is Sasakian. In the second case necessarily $0<b<\frac32$, and the structure is not contact metric.
\end{theorem}

\begin{proof}
Equations \eqref{eq:SE-d} give
\begin{align}
\Phi&=-a(c\omega_3+s\omega_1),\label{eq:ansatz-Phi}\\
\dd\Phi&=-\frac{3s}{b}\eta\wedge\gamma,\label{eq:ansatz-dPhi}\\
B:=\dd\eta-2\Phi
&=2s\widehat\omega_1+2\left(c-\frac ba\right)\widehat\omega_3
=2\left(1-\frac{bc}{a}\right)\chi+\frac{2bs}{a}\beta.
\label{eq:ansatz-B}
\end{align}
The forms $\Phi$, $B$, $\chi$, $\beta$ and $\gamma$ are horizontal.
Since $\gamma$ is horizontal, \eqref{eq:ansatz-dPhi} gives
\[
\dd\Phi|_{\mathcal D}=0,\qquad
\iota_\xi\dd\Phi=-\frac{3s}{b}\gamma.
\]
Since $B$ is horizontal, condition \eqref{eq:hor1} is automatically
satisfied. Furthermore, because $\chi$ is $\varphi$-invariant whereas
$\beta$ is $\varphi$-anti-invariant, condition \eqref{eq:hor2} is
equivalent to
\begin{equation}\label{eq:ansatz-cond1}
\frac{bc}{a}=1.
\end{equation}
In particular, $c>0$. Under this condition, \eqref{eq:beta-types} gives
\[
B(\varphi X,Y)=-\frac{2bs}{a}\gamma(X,Y).
\]
Thus \eqref{eq:hor3} is equivalent to
\begin{equation}\label{eq:ansatz-cond2}
\frac{3s}{b}=\frac{2bs}{a}.
\end{equation}

If $s=0$, then $c=1$ by positivity and \eqref{eq:ansatz-cond1} gives $a=b$. This is a standard $D$-homothetic deformation; a direct connection verification is given below in Remark~\ref{rem:Dhom}.

If $s\neq0$, equation \eqref{eq:ansatz-cond2} yields $a=\frac23b^2$, and \eqref{eq:ansatz-cond1} yields $c=\frac23b$. Since $0<c<1$, one has $0<b<\frac32$. Conversely, these equalities imply all the conditions of Proposition~\ref{prop:horizontal-test}. Finally, \eqref{eq:ansatz-B} reduces to
$B=\frac{2bs}{a}\beta\neq0$, so the structure is not contact metric.
\end{proof}

\begin{remark}\label{rem:phase}
The theorem characterizes only the constant-coefficient class \eqref{eq:ansatz}; it is not a classification of all quasi-contact metric structures. A general constant rotation $c\varphi_3+u\varphi_1+v\varphi_2$, with $c^2+u^2+v^2=1$, reduces to the same calculation by a constant rotation of $(\varphi_1,\varphi_2)$. Such a rotation preserves \eqref{eq:SE-d}. When the characteristic flow is complete, its action \eqref{eq:SE-Lie} identifies these constant phase choices.
\end{remark}


\medskip

Now take
\begin{equation}\label{eq:parameters}
\lambda=\sqrt{1+\mu^2},\qquad
b=\frac{3}{2\lambda},\qquad
 a=\frac{3}{2\lambda^2},\qquad
c=\frac1\lambda,\qquad s=\frac\mu\lambda.
\end{equation}
Then \eqref{eq:ansatz} is precisely \eqref{eq:main-eta}--\eqref{eq:main-g}. For $\mu>0$, it belongs to the second condition of Theorem~\ref{thm:ansatz}. For $\mu=0$, it belongs to the first condition with $a=b=\frac32$.

The forms \eqref{eq:beta-gamma} now satisfy
\begin{equation}\label{eq:family-dforms}
\dd\eta_\mu-2\Phi_\mu=2\mu\beta,
\qquad
\dd\Phi_\mu=-2\mu\eta_\mu\wedge\gamma.
\end{equation}
The complex form
$\Psi=(\dd t+i\eta_\mu)\wedge(\beta+i\gamma)$
is of type $(3,0)$ for the cone almost complex structure. Substitution in \eqref{eq:cone-dF} gives the explicit formula
\begin{equation}\label{eq:family-cone-dF}
\dd\widetilde\Phi_\mu
=2\mu e^{-2t}\Rea\bigl((\dd t+i\eta_\mu)
\wedge(\beta+i\gamma)\bigr).
\end{equation}
This proves the quasi-contact condition directly from the definition. For $\mu>0$, this three-form is nonzero: its contraction with $T$ is $2\mu e^{-2t}\beta$. The cone is consequently not almost K\"ahler, and Lemma~\ref{lem:NK-cone} shows that it is not nearly K\"ahler either.

\begin{remark}
The family is smooth in the parameter $\mu$, including at $\mu=0$. It can equally be written using $\lambda>1$ and $\mu=\sqrt{\lambda^2-1}$; the parameter $\mu$ is preferable when discussing the Sasakian endpoint. Negative values of $\mu$ are also permitted and differ by the constant phase change $(\varphi_1,\varphi_2)\mapsto(-\varphi_1,-\varphi_2)$.
\end{remark}

\section{Covariant identities and curvature properties}\label{sec:covariant}

We now study the family constructed above from a covariant point of
view. After computing the Levi-Civita connection of the $2$-scale
deformation, we derive explicit differential identities for the
structure tensors. Besides providing an independent verification of the
quasi-contact condition, these formulas lead to the curvature and
completeness properties of the family.


We first compute the connection for arbitrary $a,b>0$ in \eqref{eq:ansatz}. Set
\[
k=\frac{b^2-a}{a}.
\]

\begin{lemma}\label{lem:deformed-LC}
The Levi-Civita connections of $g_0$ and $g$ are related by
\begin{equation}\label{eq:deformed-LC}
\nabla_XY=\nabla^0_XY-k\{\eta_0(X)\varphi_3Y+
\eta_0(Y)\varphi_3X\}.
\end{equation}
Moreover,
\begin{align}
\nabla_X\xi&=-\frac ba\varphi_3X,\label{eq:deformed-xi}\\
(\nabla_X\varphi_3)Y
&=\frac ba\{g(X,Y)\xi-\eta(Y)X\},\label{eq:deformed-phi3}\\
(\nabla_X\varphi_1)Y
&=(1-2k)\eta_0(X)\varphi_2Y
-(1+k)\eta_0(Y)\varphi_2X+\omega_2(X,Y)\xi_0,
\label{eq:deformed-phi1}\\
(\nabla_X\varphi_2)Y
&=-(1-2k)\eta_0(X)\varphi_1Y
+(1+k)\eta_0(Y)\varphi_1X-\omega_1(X,Y)\xi_0.
\label{eq:deformed-phi2}
\end{align}
\end{lemma}

\begin{proof}
Write $C_XY=\nabla_XY-\nabla^0_XY$. Since
$\nabla^0_X\eta_0=-\omega_3(X,\cdot)$, differentiating the metric gives
\[
(\nabla^0_Xg)(Y,Z)=-(b^2-a)
\{\omega_3(X,Y)\eta_0(Z)+\eta_0(Y)\omega_3(X,Z)\}.
\]
The difference formula for two Levi-Civita connections reads
\begin{align*}
2g(C_XY,Z)
={}&(\nabla^0_Xg)(Y,Z)+(\nabla^0_Yg)(X,Z)
-(\nabla^0_Zg)(X,Y)\\
={}&-2(b^2-a)\{\eta_0(Y)\omega_3(X,Z)
+\eta_0(X)\omega_3(Y,Z)\}.
\end{align*}
Since $g=a g_0$ on $\D$, this is \eqref{eq:deformed-LC}. Inserting $Y=\xi_0$ gives
$\nabla_X\xi_0=-(1+k)\varphi_3X$ and hence \eqref{eq:deformed-xi}.
For any endomorphism $F$,
\[
(\nabla_XF)Y=(\nabla^0_XF)Y+C_X(FY)-F(C_XY).
\]
For $F=\varphi_3$, the correction is
\[
C_X(\varphi_3Y)-\varphi_3(C_XY)
=-k\eta_0(Y)X+k\eta_0(X)\eta_0(Y)\xi_0.
\]
Combining this with \eqref{eq:SE-phi3} and
$g/a=g_0+k\eta_0\otimes\eta_0$ gives \eqref{eq:deformed-phi3}.
For $F=\varphi_1$, the quaternionic relations give
\[
C_X(\varphi_1Y)-\varphi_1(C_XY)
=-2k\eta_0(X)\varphi_2Y-k\eta_0(Y)\varphi_2X.
\]
Together with \eqref{eq:SE-phi1}, this proves \eqref{eq:deformed-phi1}. Repeating the same substitution for $\varphi_2$ gives \eqref{eq:deformed-phi2}.
\end{proof}

\begin{remark}\label{rem:Dhom}
When $a=b$ and $\varphi=\varphi_3$, the coefficient $b/a$ in \eqref{eq:deformed-phi3} is one. Thus the first branch of Theorem~\ref{thm:ansatz} is Sasakian directly by \eqref{eq:Sasakian}.
\end{remark}


For the family \eqref{eq:parameters}, $k=\frac12$ and $b/a=\lambda$. Hence the connection is independent of $\mu$:
\begin{equation}\label{eq:family-LC}
\nabla_XY=\nabla^0_XY-\frac12
\{\eta_0(X)\varphi_3Y+\eta_0(Y)\varphi_3X\}.
\end{equation}
In the following calculations, we suppress the subscript $\mu$
when no confusion can arise.
Because $\Lie_{\xi_0}\varphi_1=3\varphi_2$ and $\Lie_{\xi_0}\varphi_3=0$, one has
\begin{equation}\label{eq:family-h-computation}
h=\frac12\Lie_{(2\lambda/3)\xi_0}
\left(\frac{\varphi_3+\mu\varphi_1}{\lambda}\right)
=\mu\varphi_2.
\end{equation}
In particular $h$ is skew-adjoint for $g$, annihilates $\xi$, anticommutes with $\varphi$, and satisfies \eqref{eq:main-h}.
Equations \eqref{eq:deformed-phi3}--\eqref{eq:deformed-phi2} simplify to
\begin{align}
(\nabla_X\varphi_3)Y
&=\lambda\{g(X,Y)\xi-\eta(Y)X\},\label{eq:family-phi3}\\
(\nabla_X\varphi_1)Y
&=\lambda\{g(\varphi_2X,Y)\xi-\eta(Y)\varphi_2X\},\label{eq:family-phi1}\\
(\nabla_X\varphi_2)Y
&=-\lambda\{g(\varphi_1X,Y)\xi-\eta(Y)\varphi_1X\}.
\label{eq:family-phi2}
\end{align}
For instance, in \eqref{eq:deformed-phi1} the first term vanishes, while
$\omega_2(X,Y)\xi_0=\lambda g(\varphi_2X,Y)\xi$ and
$\frac32\eta_0(Y)=\lambda\eta(Y)$.
Combining \eqref{eq:family-phi3} and \eqref{eq:family-phi1} proves
\begin{equation}\label{eq:fundamental-derivative}
(\nabla_X\varphi)Y
=g(X+hX,Y)\xi-\eta(Y)(X+hX),
\end{equation}
which is \eqref{eq:main-derivative}. Furthermore,
\begin{equation}\label{eq:family-nabla-xi}
\nabla_X\xi=-\lambda\varphi_3X=-\varphi X-\varphi hX,
\qquad
\nabla_\xi\varphi=0.
\end{equation}
The equality $\varphi(\Id+h)=\lambda\varphi_3$ follows immediately from \eqref{eq:quaternion} and $\lambda^2=1+\mu^2$.


\medskip

We now give a second verification of the quasi-contact condition,
working directly with the Levi-Civita connection of the cone rather
than with exterior forms.

Equations \eqref{eq:fundamental-derivative} and \eqref{eq:family-nabla-xi}, substituted into \eqref{eq:cone-nabla1}--\eqref{eq:cone-nabla2}, give
\begin{align}
(\widetilde\nabla_X\widetilde J)Y
&=g(hX,Y)\xi-\eta(Y)hX+g(\varphi hX,Y)T,
\label{eq:explicit-cone-nabla}\\
(\widetilde\nabla_X\widetilde J)T&=-\varphi hX,
\qquad \widetilde\nabla_T\widetilde J=0.
\label{eq:explicit-cone-nablaT}
\end{align}
These formulas give a covariant verification of quasi-K\"ahlerity without invoking Lemma~\ref{lem:QK-type} or an external characterization theorem.
Indeed, $h\varphi=-\varphi h$ and $\varphi h\varphi=h$ imply
\[
g(h\varphi X,\varphi Y)=-g(hX,Y),\qquad
g(\varphi h\varphi X,\varphi Y)=-g(\varphi hX,Y).
\]
Using $\widetilde JX=\varphi X-\eta(X)T$ and $\widetilde\nabla_T\widetilde J=0$, equations \eqref{eq:explicit-cone-nabla}--\eqref{eq:explicit-cone-nablaT} yield
\[
(\widetilde\nabla_{\widetilde JX}\widetilde J)\widetilde JY
=-g(hX,Y)\xi+\eta(Y)hX-g(\varphi hX,Y)T.
\]
This is the negative of \eqref{eq:explicit-cone-nabla}. For the second argument $T$, the two terms are $-\varphi hX$ and $\varphi hX$. Finally, for first argument $T$, both terms vanish because $h\xi=0$. Tensoriality proves \eqref{eq:QK} for arbitrary arguments.
For comparison with \cite[Theorem~2.1]{BJPS}, equation \eqref{eq:fundamental-derivative} also gives
\begin{equation}\label{eq:intrinsic-QC}
\begin{aligned}
(\nabla_X\varphi)Y+(\nabla_{\varphi X}\varphi)\varphi Y
={}&2g(X,Y)\xi-\eta(Y)X-\eta(X)\eta(Y)\xi-\eta(Y)hX.
\end{aligned}
\end{equation}
To see this, note that
\[
g(\varphi X+h\varphi X,\varphi Y)
=g(X,Y)-\eta(X)\eta(Y)-g(hX,Y)
\]
and apply \eqref{eq:fundamental-derivative} to the pair $(\varphi X,\varphi Y)$. Thus the examples satisfy exactly the intrinsic identity used in that paper as well.


\medskip

We now turn to the curvature of the deformed structures and to the
Sasakian structure naturally associated with the family.
Let $S(X)=\nabla_X\xi=-\lambda\varphi_3X$. Since $\lambda$ is constant, \eqref{eq:family-phi3} gives $\nabla_\xi S=0$, while
\[
S^2=-\lambda^2(\Id-\eta\otimes\xi).
\]
For a unit geodesic vector field, direct expansion of the curvature gives
$R(X,\xi)\xi=-(\nabla_\xi S)X-S^2X$. Hence
\begin{equation}\label{eq:family-Jacobi}
R(X,\xi)\xi=\lambda^2\{X-\eta(X)\xi\}.
\end{equation}
This proves \eqref{eq:main-Jacobi}. In particular, every $2$-plane containing $\xi$ has sectional curvature $\lambda^2=1+\mu^2$, and
\begin{equation}\label{eq:family-Ric}
\Ric(\xi,\xi)=4(1+\mu^2).
\end{equation}
The family also has an associated Sasakian structure. Namely,
\begin{equation}\label{eq:companion}
\eta_*:=\lambda\eta_\mu=\frac32\eta_0,
\qquad \xi_*:=\lambda^{-1}\xi_\mu=\frac23\xi_0,
\qquad g_*:=\lambda^2g_\mu=\frac32g_0+\frac34\eta_0^2
\end{equation}
are independent of $\mu$, and $(\varphi_3,\xi_*,\eta_*,g_*)$ is Sasakian by \eqref{eq:family-phi3}. Here and below $\eta_0^2$ abbreviates $\eta_0\otimes\eta_0$. 
Although the metrics $g_\mu$ are all homothetic, the corresponding
almost contact metric structures are pairwise non-isomorphic for
distinct values of $\mu\geq0$. Indeed,
\begin{equation}\label{eq:h-norm}
\|h_\mu\|_{\HS}^2=4\mu^2,
\end{equation}
and the pointwise Hilbert-Schmidt norm of
$h=\frac12\Lie_\xi\varphi$ is invariant under isomorphisms of almost
contact metric structures. Hence an isomorphism between the structures
corresponding to $\mu$ and $\nu$ would imply $\mu=\nu$.

\medskip

The eigenvalues of $g_\mu$ relative to $g_0$ are the positive constants
$\frac{3}{2\lambda^2}$ on $\D$ and $\frac{9}{4\lambda^2}$ on $\R\xi_0$. Thus
\[
\frac{3}{2\lambda^2}g_0\leq g_\mu\leq
\frac{9}{4\lambda^2}g_0.
\]
The two induced distance functions are uniformly equivalent. If $g_0$ is complete, any $g_\mu$-Cauchy sequence is $g_0$-Cauchy and converges, and the upper bound implies convergence in $g_\mu$ as well. Hence $g_\mu$ is complete. This proves the remaining assertion of Theorem~\ref{thm:main}.


\section{An explicit homogeneous example on the sphere}\label{sec:sphere}


Consider the unit sphere $S^5\subset\C^3$, with round metric $g_0$ and standard complex structure $I$. Let $N_p=p$ be the outward unit normal and put
\begin{equation}\label{eq:sphere-data}
\xi_0=IN,\qquad \eta_0=g_0(\xi_0,\cdot),\qquad
\varphi_3X=-IX-\eta_0(X)N.
\end{equation}
The last expression is tangent to the sphere and annihilates $\xi_0$. Define
\begin{equation}\label{eq:sphere-volume}
\Omega=\dd z_1\wedge\dd z_2\wedge\dd z_3,
\qquad \alpha_p(X,Y)=\Omega(N_p,X,Y),
\end{equation}
and let $\omega_1=\Rea\alpha$, $\omega_2=-\Ima\alpha$. Define $\varphi_1,\varphi_2$ by $g_0(\varphi_iX,Y)=\omega_i(X,Y)$.
All these tensors are invariant under the transitive $\SU(3)$-action on $S^5$. To verify their algebraic relations explicitly, take $p=(0,0,1)$. On $\D_p$, in the real coordinates $(x_1,y_1,x_2,y_2)$,
\begin{align}
\omega_1&=\dd x_1\wedge\dd x_2-\dd y_1\wedge\dd y_2,\label{eq:sphere-forms1}\\
\omega_2&=-\dd x_1\wedge\dd y_2-\dd y_1\wedge\dd x_2,\label{eq:sphere-forms2}\\
\omega_3&=-\dd x_1\wedge\dd y_1-\dd x_2\wedge\dd y_2.
\label{eq:sphere-forms3}
\end{align}
For example,
\[
\varphi_1\partial_{x_1}=\partial_{x_2},\qquad
\varphi_2\partial_{x_1}=-\partial_{y_2},\qquad
\varphi_3\partial_{x_1}=-\partial_{y_1}.
\]
The remaining basis vectors are treated in the same way, and \eqref{eq:quaternion} follows.

The Euclidean Gauss formula and parallelism of $\Omega$ give
\[
(\nabla^0_X\alpha)(Y,Z)=\Omega(X,Y,Z)
=i(\eta_0\wedge\alpha)(X,Y,Z),
\]
exactly as in \eqref{eq:alpha-tangent}. Alternatively, extending $N$ radially gives
$\dd\alpha=(\Lie_N\Omega)|_{TS^5}=3\Omega|_{TS^5}$.
These identities verify all the differential equations \eqref{eq:SE-d} and \eqref{eq:SE-xi}--\eqref{eq:SE-phi2} directly on the sphere.


Choose
\[
\lambda=\frac32,\qquad \mu=\frac{\sqrt5}{2}.
\]
The formulas of Theorem~\ref{thm:main} become
\begin{equation}\label{eq:sphere-final}
\eta=\eta_0,\qquad \xi=\xi_0,\qquad
 g=\frac23g_0+\frac13\eta_0^2,\qquad
\varphi=\frac23\varphi_3+\frac{\sqrt5}{3}\varphi_1.
\end{equation}
In this normalization the contact form and its Hopf field are unchanged. The characteristic flow is $p\mapsto e^{iu}p$ and preserves $g$. On the other hand,
\begin{equation}\label{eq:sphere-h-final}
h=\frac{\sqrt5}{2}\varphi_2,\qquad
h^2=-\frac54(\Id-\eta\otimes\xi),
\qquad \nabla_X\xi=-\frac32\varphi_3X.
\end{equation}
The fundamental form and the exterior derivative of the contact form are
\begin{equation}\label{eq:sphere-defect}
\Phi=-\frac49\omega_3-\frac{2\sqrt5}{9}\omega_1,
\qquad \dd\eta=-2\omega_3.
\end{equation}
Since $\omega_1$ and $\omega_3$ are linearly independent at every point, $\dd\eta$ and $\Phi$ are not even proportional. Thus failure of the contact metric equation cannot be attributed to a convention concerning a factor of two.

By \eqref{eq:family-Jacobi}, every $2$-plane containing $\xi$ has sectional curvature $9/4$.
The metric and all the structure tensors are $\SU(3)$-invariant. The structure is therefore homogeneous, while the full Hopf circle preserves the metric but rotates the additional endomorphisms. The metric is complete since the sphere is compact; alternatively, use the uniform comparison with $g_0$ proved above. This completes the direct verification of Corollary~\ref{cor:sphere-intro}.

\section{Quasi-contact structures with Killing characteristic field}\label{sec:Killing}

We now work in arbitrary dimension $2n+1$. All the statements in this section are pointwise or local and require no compactness assumption.


\begin{proposition}\label{prop:Killing-algebra}
On a quasi-contact metric manifold, $\xi$ is Killing if and only if $h$ is skew-adjoint. Under this condition, set $A=-\nabla\xi=\varphi(\Id+h)$. Then, on $\D$,
\begin{equation}\label{eq:A-algebra}
A^*=-A,\qquad A^2=-\Id+h^2,\qquad Ah=-hA.
\end{equation}
The structure is contact metric if and only if $h=0$.
\end{proposition}

\begin{proof}
If $\xi$ is Killing, Lie differentiation of
$g(\varphi X,Y)+g(X,\varphi Y)=0$ gives
$g(hX,Y)+g(X,hY)=0$. Conversely, if $h^*=-h$, then
\[
(\varphi h)^*=h^*\varphi^*=h\varphi=-\varphi h,
\]
so \eqref{eq:h-nabla-xi} makes $\nabla\xi$ skew-adjoint and $\xi$ is Killing. The first identity of \eqref{eq:A-algebra} follows. Since $h\varphi=-\varphi h$, on $\D$ we have
\[
A^2=\varphi(\Id+h)\varphi(\Id+h)
=(-\Id+h)(\Id+h)=-\Id+h^2.
\]
Similarly,
$Ah=\varphi h+\varphi h^2=-hA$.
Finally, Lemma~\ref{lem:h-identities} says that contact metric compatibility is equivalent to self-adjointness of $h$. Under the present skew-adjointness assumption this is equivalent to $h=0$.
\end{proof}

\begin{proposition}\label{prop:quaternionic-blocks}
At any point of a quasi-contact metric manifold with Killing $\xi$, the orthogonal decomposition of $\D$ into eigenspaces of $-h^2$ is invariant under $\varphi,h,A$. On an eigenspace $E_\rho$ with eigenvalue $\rho^2>0$, the endomorphisms
\[
J_1=\varphi|_{E_\rho},\qquad
J_2=\rho^{-1}h|_{E_\rho},\qquad J_3=J_1J_2
\]
are orthogonal complex structures satisfying the quaternionic relations. Hence $\dim_\R E_\rho$ is divisible by four. On this subspace,
\[
A^2=-(1+\rho^2)\Id.
\]
On $\ker h\cap\D$, one has $A=\varphi$.
\end{proposition}

\begin{proof}
The operator $-h^2$ is self-adjoint and nonnegative. Its commutation with $\varphi$ follows from $h\varphi=-\varphi h$, and its commutation with $A$ follows from \eqref{eq:A-algebra}. On $E_\rho$, both $J_1$ and $J_2$ square to $-\Id$ and anticommute; they are orthogonal because they are skew-adjoint complex structures. Their product is again skew-adjoint and squares to $-\Id$.

For a unit vector $v\in E_\rho$, the vectors
$v,J_1v,J_2v,J_3v$ are orthonormal. Their span is invariant under the three endomorphisms, and its orthogonal complement is invariant as well. Repeating the construction gives an orthogonal decomposition into $4$-dimensional blocks. The statements about $A^2$ and $\ker h$ follow from its definition and \eqref{eq:A-algebra}.
\end{proof}

This proposition concerns individual tangent spaces. It makes no assertion that the ranks of the eigenspaces are constant, or that an eigenframe can be chosen smoothly through an eigenvalue crossing. In particular, $\rank h$ is a multiple of four at every point. For $\mu>0$, the family constructed above realizes a single nonzero block, with $\rho=\mu$.

The argument only forces $\ker h_p\cap\mathcal D_p$ to be
$\varphi_p$-invariant, and hence even-dimensional. Therefore
$\dim\mathcal D$ need not be divisible by four. If
$h_p|_{\mathcal D_p}$ is injective, however, then
$\dim\mathcal D$ is divisible by four.


\begin{lemma}\label{lem:unit-Killing}
Let $\xi$ be a unit Killing vector field on a Riemannian manifold and put $S=\nabla\xi$. Then
\[
\nabla_\xi\xi=0,\qquad \nabla_\xi S=0,\qquad
R(X,\xi)\xi=-S^2X.
\]
\end{lemma}
\begin{proof}
Skew-adjointness of $S$ and constant length give
$g(S\xi,X)=-g(\xi,SX)=0$. The local flow of $\xi$ preserves both $g$ and $\xi$, and therefore also the Levi-Civita connection and $S=\nabla\xi$. Thus $\Lie_\xi S=0$. Directly,
\[
(\Lie_\xi S)X=(\nabla_\xi S)X-S(SX)+S(SX)
=(\nabla_\xi S)X.
\]
Finally, using $[X,\xi]=SX-\nabla_\xi X$ and $\nabla_\xi\xi=0$ in the definition of curvature gives
\[
R(X,\xi)\xi=-(\nabla_\xi S)X-S^2X=-S^2X.
\]
\end{proof}

We are now in a position to prove Theorem~\ref{thm:curvature-intro}.
The argument combines the algebraic identities for $h$ obtained above
with the curvature formula for a unit Killing vector field.

\begin{proof}[Proof of Theorem~\ref{thm:curvature-intro}]
Apply Lemma~\ref{lem:unit-Killing} and Proposition~\ref{prop:Killing-algebra}. On all of $TM$,
\[
R(X,\xi)\xi=-A^2X=X-\eta(X)\xi-h^2X.
\]
For a horizontal unit vector $X$, skew-adjointness of $h$ gives
\begin{equation}\label{eq:sectional-sharp}
K(X,\xi)=1+\|hX\|^2\geq1.
\end{equation}
Tracing over a horizontal orthonormal basis yields
\[
\Ric(\xi,\xi)=2n-\tr(h^2)
=2n+\|h\|_{\HS}^2.
\]
Equality $\Ric(\xi,\xi)=2n$ is therefore equivalent to $h=0$ at that point. The global K-contact criterion follows from Proposition~\ref{prop:Killing-algebra}. The multiplicity assertion is Proposition~\ref{prop:quaternionic-blocks}.
\end{proof}

\begin{corollary}\label{cor:corrected-rigidity}
For a quasi-contact metric manifold with Killing characteristic vector field, the following conditions are equivalent:
\begin{enumerate}[label=\textup{(\roman*)},leftmargin=2.6em]
\item the structure is K-contact;
\item $\Lie_\xi\varphi=0$;
\item $\Ric(\xi,\xi)=2n$;
\item every $2$-plane containing $\xi$ has sectional curvature $1$.
\end{enumerate}
It is enough in \textup{(iv)} to require that all these sectional curvatures be at most $1$.
\end{corollary}
\begin{proof}
Use $\Lie_\xi\varphi=2h$, Proposition~\ref{prop:Killing-algebra}, and \eqref{eq:sectional-sharp}.
\end{proof}

The constructed family has $K(X,\xi)=1+\mu^2$ on every characteristic plane and $\Ric(\xi,\xi)=4+4\mu^2$. It demonstrates that the lower bound in \eqref{eq:sectional-sharp} is sharp but need not be attained. In particular, the Killing condition does not force preservation of $\varphi$.

\section{Products of cones and examples in higher dimensions}\label{sec:joins}

The $5$-dimensional examples give strict structures in every higher odd dimension by taking products of cones. We give the induced tensors explicitly, including the characteristic field and its Lie derivative tensor. The resulting examples still have Killing characteristic vector field, but their characteristic Jacobi operators have two distinct eigenvalues on the horizontal bundle.


\medskip

Write $E=r\partial_r$ for the Euler field of a Riemannian cone $(C(L),G)=(L\times\R_{>0},\dd r^2+r^2g_L)$. An almost Hermitian structure $J$ is called dilation-invariant if $\Lie_EJ=0$. The cone structure associated with any almost contact metric structure has this property. In radial coordinates, our sign convention is
\begin{equation}\label{eq:radial-cone-J}
JX=\varphi X+\eta(X)E,\qquad JE=-\xi.
\end{equation}

\begin{lemma}\label{lem:homogeneous-cone-link}
Let $J$ be a dilation-invariant almost complex structure on $C(L)$
which is compatible with the cone metric $G$, that is,
$G(JU,JV)=G(U,V)$
for all vector fields $U,V$ on $C(L)$. On the slice $r=1$, define
\[
\xi=-JE,\qquad \eta(X)=g_L(\xi,X),\qquad
\varphi X=JX-\eta(X)E.
\]
These tensors give an almost contact metric structure on $L$, whose cone is the given $(C(L),J,G)$. The structure on $L$ is quasi-contact metric precisely when $(J,G)$ is quasi-K\"ahler, and contact metric precisely when $(J,G)$ is almost K\"ahler.
\end{lemma}
\begin{proof}
On the slice, $E$ is a unit normal. The compatibility of $J$ with the cone metric $G$
shows that $\xi=-JE$ is tangent and unit, and that
\[
G(JX,E)=-G(X,JE)=\eta(X).
\]
Thus, $\varphi X$ is tangent. The equalities $J^2=-\Id$ and $J\xi=E$ give
\[
\varphi\xi=0,\qquad \eta\circ\varphi=0,\qquad
\varphi^2=-\Id+\eta\otimes\xi.
\]
Taking inner products yields the almost contact metric compatibility identity. Dilation invariance extends these equalities from $r=1$ to \eqref{eq:radial-cone-J} on the whole cone. The two equivalences now follow from the definition and \eqref{eq:cone-dF}.
\end{proof}

\begin{lemma}\label{lem:product-QK}
The product of two almost Hermitian manifolds is quasi-K\"ahler if and only if each factor is quasi-K\"ahler. It is almost K\"ahler if and only if each factor is almost K\"ahler.
\end{lemma}
\begin{proof}
For the product metric and product almost complex structure, the Levi-Civita connection splits. Hence the quasi-K\"ahler expression in \eqref{eq:QK} is the ordered pair of the corresponding expressions on the factors. Necessity follows by choosing arguments tangent to one factor. If $F_i$ are the factor fundamental forms and $\pi_i$ the projections, the product form is $F=\pi_1^*F_1+\pi_2^*F_2$. Therefore
\[
\dd F=\pi_1^*\dd F_1+\pi_2^*\dd F_2.
\]
Restriction to either factor shows that this is zero precisely when both summands are zero.
\end{proof}


\medskip

We now describe explicitly the almost contact metric structure induced
on the link of the product of two cones. We use the metric identification in \cite[Example~3.3]{ACGL}
and determine the almost contact metric structure induced by the product almost complex structure.  
Let $(M_i,\varphi^{(i)},\xi_i,\eta_i,g_i)$, $i=1,2$, be almost contact metric manifolds. The superscripts on $\varphi^{(i)}$ label the two factors and must not be confused with the quaternionic indices used earlier.

The hypersurface
\[
L=\{r_1^2+r_2^2=1\}
\subset C(M_1)\times C(M_2),
\qquad r_1,r_2>0,
\]
is the link of the product cone. 
Set
\[
f(\theta)=\cos\theta,\qquad s(\theta)=\sin\theta.
\]
Writing
$r_1=\cos\theta$ and $r_2=\sin\theta$ identifies $L$ with
\[
M_1\times M_2\times(0,\pi/2),
\]
equipped with the metric
\begin{equation}\label{eq:join-metric}
g_L=\dd\theta^2+f^2g_1+s^2g_2.
\end{equation}
All factor tensors below are pulled back to $L$. 
We shall refer to this Riemannian manifold as the
\emph{open spherical join} of the two factors.

Put
\begin{equation}\label{eq:join-xi-eta}
\xi_L=\xi_1+\xi_2,\qquad
\eta_L=f^2\eta_1+s^2\eta_2,\qquad
Z=\tan\theta\,\xi_1-\cot\theta\,\xi_2.
\end{equation}
For $X_i\in TM_i$ and $u\in\R$, define
\begin{equation}\label{eq:join-phi}
\begin{split}
\varphi_L(X_1+X_2+u\partial_\theta)
={}&\varphi^{(1)}X_1+\varphi^{(2)}X_2+uZ+fs\{\eta_2(X_2)-\eta_1(X_1)\}\partial_\theta.
\end{split}
\end{equation}
The fields $\partial_\theta$ and $Z$ are orthonormal, horizontal for $\eta_L$, and satisfy
\[
\varphi_L\partial_\theta=Z,\qquad \varphi_LZ=-\partial_\theta.
\]
In particular, with $\D_i=\ker\eta_i$, the horizontal bundle splits orthogonally as
\begin{equation}\label{eq:join-horizontal}
\ker\eta_L=\D_1\oplus\D_2\oplus
\Span\{\partial_\theta,Z\}.
\end{equation}

\begin{theorem}\label{thm:open-join}
The tensors \eqref{eq:join-metric}--\eqref{eq:join-phi} define an almost contact metric structure on $L$. This structure is quasi-contact metric if and only if both factors are quasi-contact metric. It is contact metric if and only if both factors are contact metric. In particular, 
if both factors are quasi-contact metric and $h_i=\frac12\Lie_{\xi_i}\varphi^{(i)}$, then
\begin{equation}\label{eq:join-h}
h_L(X_1+X_2+u\partial_\theta)=h_1X_1+h_2X_2.
\end{equation}
Moreover, $\xi_L$ is Killing if and only if both $\xi_i$ are Killing. The metric $g_L$ is geodesically incomplete.
\end{theorem}
\begin{proof}
Let $r_i$ be the radial coordinates of the two cones and set
\[
r_1=\rho\cos\theta,\qquad r_2=\rho\sin\theta.
\]
This is a diffeomorphism from $\R_{>0}\times(0,\pi/2)$ onto the positive quadrant. A direct calculation gives
\begin{equation}\label{eq:product-cone-metric}
\dd r_1^2+r_1^2g_1+\dd r_2^2+r_2^2g_2
=\dd\rho^2+\rho^2g_L.
\end{equation}
The product almost complex structure $J=J_1\oplus J_2$ is invariant under the simultaneous dilations generated by
\[
E=\rho\partial_\rho=E_1+E_2,
\qquad E_i=r_i\partial_{r_i}.
\]
Since $J_iE_i=-\xi_i$, Lemma~\ref{lem:homogeneous-cone-link} gives $\xi_L=\xi_1+\xi_2$, and the metric gives the displayed formula for $\eta_L$.
To compute the tangential endomorphism, observe that
\[
E_1=f^2E-fs\partial_\theta,\qquad
E_2=s^2E+fs\partial_\theta.
\]
The factor formula $J_iX_i=\varphi^{(i)}X_i+\eta_i(X_i)E_i$ therefore yields
\[
J(X_1+X_2)=\varphi^{(1)}X_1+\varphi^{(2)}X_2
+fs\{\eta_2(X_2)-\eta_1(X_1)\}\partial_\theta
+\eta_L(X_1+X_2)E.
\]
Also $J\partial_\theta=Z$. Taking the tangential part proves \eqref{eq:join-phi}. The almost contact metric identities and the two equivalences now follow from Lemmas~\ref{lem:homogeneous-cone-link} and \ref{lem:product-QK}.
For the Lie derivative tensor, quasi-contact geometry gives $\Lie_{\xi_i}\eta_i=0$. The lifted fields commute with the other factor and with $\partial_\theta$, so that
\[
[\xi_L,\partial_\theta]=[\xi_L,Z]=0,
\qquad \xi_L(\theta)=0.
\]
Taking the Lie derivative of \eqref{eq:join-phi} thus leaves only $\Lie_{\xi_i}\varphi^{(i)}=2h_i$, proving \eqref{eq:join-h}. Furthermore,
\[
\Lie_{\xi_L}g_L=f^2\Lie_{\xi_1}g_1+s^2\Lie_{\xi_2}g_2.
\]
The summands act on different factor tangent spaces, so that their sum vanishes exactly when both vanish.
Finally, fixing points $p_i\in M_i$ gives a curve $\theta\mapsto(p_1,p_2,\theta)$ with unit speed. The Koszul formula gives $\nabla^L_{\partial_\theta}\partial_\theta=0$, so this curve is a geodesic. It reaches either excluded endpoint of $(0,\pi/2)$ in finite parameter time and has no limit in $L$. Thus $g_L$ is geodesically incomplete.
\end{proof}


The proof of Theorem~\ref{thm:open-join} uses the almost Hermitian
geometry of the product of the corresponding cones. We now verify
the contact metric condition directly on the open spherical join by
computing 
$\dd\eta_L-2\Phi_L$.

Write $\Phi_i=g_i(\cdot,\varphi^{(i)}\cdot)$ and $B_i=\dd\eta_i-2\Phi_i$. From \eqref{eq:join-phi},
\begin{equation}\label{eq:join-Phi}
\Phi_L=f^2\Phi_1+s^2\Phi_2
+fs\,\dd\theta\wedge(\eta_2-\eta_1).
\end{equation}
Differentiating \eqref{eq:join-xi-eta} gives
\[
\dd\eta_L=f^2\dd\eta_1+s^2\dd\eta_2
+2fs\,\dd\theta\wedge(\eta_2-\eta_1).
\]
Consequently,
\begin{equation}\label{eq:join-defect}
\dd\eta_L-2\Phi_L=f^2B_1+s^2B_2.
\end{equation}
There is no cancellation between the two terms: their restrictions to $TM_1$ and $TM_2$ recover $f^2B_1$ and $s^2B_2$, respectively. In particular, if the contact metric equation fails at every point
of one factor, then it also fails at every point of the open join. This verifies strictness directly in the almost contact metric notation, independently of closedness of the cone form.


\medskip

We now apply the open spherical join construction to the
$5$-dimensional strict quasi-contact structures obtained above and
to the standard Sasakian spheres. This produces explicit strict
quasi-contact metric structures in every odd dimension greater than
$5$.
Take as the first factor one of the structures $(M^5,\varphi_\mu,\xi_\mu,\eta_\mu,g_\mu)$ of Theorem~\ref{thm:main}, with $\mu>0$, and as the second factor the unit sphere $S^{2q-1}\subset\C^q$, $q\geq1$, with its standard cone-induced structure $(\varphi_S,\xi_S,\eta_S,g_S)$. With the standard complex structure $I$ and outward unit normal $N$, our conventions are
\[
\xi_S=IN,\qquad \eta_S=g_S(\xi_S,\cdot),\qquad
\varphi_SX=-IX-\eta_S(X)N.
\]
Its cone is flat K\"ahler, with complex structure $-I$, and $h_S=0$. For $q=1$, the sphere is $S^1$, $\varphi_S=0$ and its horizontal bundle is zero; the cone is still the punctured complex line and all the formulas remain valid.

\begin{corollary}\label{cor:all-dimensions}
For every $q\geq1$ and $\mu>0$, the manifold
\[
L_{\mu,q}=M^5\times S^{2q-1}\times(0,\pi/2)
\]
admits a strict quasi-contact metric structure of dimension $5+2q$ with
\begin{align}
 g_L&=\dd\theta^2+\cos^2\theta\,g_\mu+\sin^2\theta\,g_S,
 \label{eq:all-dim-g}\\
 \eta_L&=\cos^2\theta\,\eta_\mu+\sin^2\theta\,\eta_S,
 &\xi_L&=\xi_\mu+\xi_S,
 \label{eq:all-dim-eta}
\end{align}
and $\varphi_L$ given by \eqref{eq:join-phi}. Its characteristic vector field is Killing, its metric is geodesically incomplete, and $h_L$ has rank $4$ everywhere. 
\end{corollary}
\begin{proof}
Apply Theorem~\ref{thm:open-join}. Equation \eqref{eq:join-defect} becomes
\[
\dd\eta_L-2\Phi_L
=\cos^2\theta\,(\dd\eta_\mu-2\Phi_\mu),
\]
which is nonzero at every point by \eqref{eq:family-dforms}. On the splitting \eqref{eq:join-horizontal}, equation \eqref{eq:join-h} says that $h_L$ equals $h_\mu$ on $\D_\mu$ and vanishes on $\D_S\oplus\Span\{\partial_\theta,Z\}$. Since $h_\mu^2=-\mu^2\Id_{\D_\mu}$, its rank is $4$.
\end{proof}

The characteristic Jacobi operator can also be computed explicitly. Let $\Pi_\mu$ denote the $g_L$-orthogonal projection onto the lifted subbundle $\D_\mu$. Then
\begin{equation}\label{eq:join-Jacobi}
h_L^2=-\mu^2\Pi_\mu,\qquad
R^L(X,\xi_L)\xi_L
=X-\eta_L(X)\xi_L+\mu^2\Pi_\mu X.
\end{equation}
The second equality follows from Theorem~\ref{thm:curvature-intro}. On $\ker\eta_L$, this operator has eigenvalues
$1+\mu^2$  with multiplicity $4$ and
$1$ with multiplicity $2q$.
In particular,
\[
\Ric_{g_L}(\xi_L,\xi_L)=4+2q+4\mu^2.
\]
Because $q\geq1$ and $\mu>0$, the Jacobi operator is not a scalar multiple of the identity on the horizontal bundle.

\medskip

Notice that the excluded endpoint $\theta=0$ is an artifact of using polar coordinates in the flat factor and can be filled smoothly. This gives a convenient alternative realization of the examples, including simply connected ones when $q=1$.

\begin{proposition}\label{prop:ball-model}
Let $B^{2q}=\{z\in\C^q:|z|<1\}$ and put $a(z)=1-|z|^2$. On $M^5\times B^{2q}$, define
\begin{align}
\widehat g&=a(z)g_\mu+|\dd z|^2+
\frac{\langle z,\dd z\rangle^2}{a(z)},
\label{eq:ball-metric}\\
\widehat\eta(X,V)&=a(z)\eta_\mu(X)+\langle Iz,V\rangle,
&\widehat\xi&=(\xi_\mu,Iz),
\label{eq:ball-eta}\\
\widehat\varphi(X,V)&=\left(
\varphi_\mu X+\frac{\langle z,V\rangle}{a(z)}\xi_\mu,
\ -IV-\widehat\eta(X,V)z\right).
\label{eq:ball-phi}
\end{align}
Here $X\in T_pM$, $V\in T_z\C^q$, $\langle\cdot,\cdot\rangle$ is the Euclidean inner product, and the square of a $1$-form denotes its tensor square. These are smooth strict quasi-contact metric structures with Killing characteristic vector field. They extend the structures of Corollary~\ref{cor:all-dimensions} across $z=0$ and remain incomplete.
\end{proposition}
\begin{proof}
Consider the product $C(M)\times\C^q$ with its product almost Hermitian structure, using $-I$ on $\C^q$. The hypersurface
\[
\Sigma=\{(p,r_1,z):r_1^2+|z|^2=1,\ r_1>0\}
\]
is identified with $M\times B^{2q}$ by $r_1=\sqrt{a(z)}$. Simultaneous dilation makes this product the cone over $\Sigma$. Since
\[
\dd r_1=-\frac{\langle z,\dd z\rangle}{\sqrt{a(z)}},
\]
restriction of the product metric gives \eqref{eq:ball-metric}. The Euler field is $E=r_1\partial_{r_1}+z$; applying $-J$ gives $\widehat\xi=(\xi_\mu,Iz)$, and its metric dual is \eqref{eq:ball-eta}.
A vector $(X,V)$ on the graph has ambient representative
$X+V-\langle z,V\rangle\partial_{r_1}/r_1$. Apply the product almost complex structure, subtract $\widehat\eta(X,V)E$, and take the $M$ and $\C^q$ components. The result is exactly \eqref{eq:ball-phi}. Lemmas~\ref{lem:homogeneous-cone-link} and \ref{lem:product-QK} show that these tensors are quasi-contact metric. They are smooth at $z=0$, where
\[
\widehat g=g_\mu+|\dd z|^2,\qquad
\widehat\varphi(X,V)=(\varphi_\mu X,-IV).
\]
Along the submanifold $M\times\{0\}$, one has
\[
(\dd\widehat\eta-2\widehat\Phi)|_{TM}
=\dd\eta_\mu-2\Phi_\mu,
\]
hence the induced structure is not contact metric there. The product flow of $\xi_\mu$ and the rotations generated by $Iz$ preserves \eqref{eq:ball-metric}, hence $\widehat\xi$ is Killing. Away from $z=0$, setting $z=\sin\theta\,u$, $u\in S^{2q-1}$, recovers the open join. Radial curves still reach $|z|=1$ in finite length, so incompleteness remains.
\end{proof}

For $M=S^5$, Proposition~\ref{prop:ball-model} gives simply connected strict examples in every dimension $5+2q$, $q\geq1$.


\medskip

We now prove that the other excluded endpoint cannot be filled by a smooth complete Riemannian extension of the same metric. 

\begin{proposition}\label{prop:join-blowup}
For the metrics in Corollary~\ref{cor:all-dimensions}, let $\sigma$ be a two-plane tangent to the $M$ factor and containing $\xi_\mu$. Then
\begin{equation}\label{eq:join-blowup}
K_{g_L}(\sigma)=1+\frac{\mu^2}{\cos^2\theta}.
\end{equation}
Thus sectional curvature is unbounded as $\theta\to\pi/2$. Neither the open join nor its ball extension can be embedded isometrically as an open subset of a complete smooth Riemannian manifold.
\end{proposition}
\begin{proof}
For a lifted vector field $X$ tangent to $M$, the Koszul formula applied to \eqref{eq:join-metric} gives
\[
\nabla^L_{\partial_\theta}X=\frac{f'}{f}X,
\qquad
\nabla^L_XY=\nabla^\mu_XY-ff'g_\mu(X,Y)\partial_\theta
\quad(X,Y\in TM).
\]
The curvature identity for a $2$-plane $\sigma\subset TM$, obtained by inserting these expressions into the definition of $R$, is
\[
K_{g_L}(\sigma)=\frac{K_{g_\mu}(\sigma)-(f')^2}{f^2}.
\]
Since $f=\cos\theta$ and $K_{g_\mu}(\sigma)=1+\mu^2$ by \eqref{eq:main-Jacobi}, this gives \eqref{eq:join-blowup}.
Suppose that such an isometric embedding into a complete Riemannian manifold existed. A unit-speed angular geodesic approaching $\theta=\pi/2$ is a finite-length Cauchy curve, so its image would converge to a point of the complete manifold. Smoothness would bound the norm of the ambient curvature tensor in a neighborhood of that point. The isometric embedding identifies this tensor with the curvature of $g_L$ on the open image, contradicting \eqref{eq:join-blowup}. The same argument applies to the ball extension, since the problematic endpoint is unchanged.
\end{proof}

This obstruction concerns the metrics just constructed. It does not rule out a different complete quasi-contact metric structure in a higher dimension or on the same underlying manifold. In particular, it does not answer the completeness question raised in the Introduction.

\section{Final remarks}\label{sec:error}

Theorem~A of \cite{BJPS} states that a quasi-contact metric manifold with Killing characteristic field is K-contact. Theorem~\ref{thm:main}, or just the explicit sphere structure \eqref{eq:sphere-final}, disproves this assertion. 
We now examine the proof of Theorem~A of \cite{BJPS} and identify the
step at which the argument fails. The identities established before
that step remain valid and are consistent with the examples constructed
above.


At a fixed point, \cite{BJPS} first chooses an orthonormal basis
\[
(\varepsilon_1,\ldots,\varepsilon_{2n})
=(e_1,e_1^*,\ldots,e_n,e_n^*)
\]
adapted to the skew-adjoint operator $\nabla\xi$ (here, $e_a^*$ is only the paired-frame notation used in \cite{BJPS}; it is not an adjoint of a vector). The frame satisfies
\begin{equation}\label{eq:BJPS-frame}
\nabla_{e_a}\xi=-\lambda_a e_a^*,\qquad
\nabla_{e_a^*}\xi=\lambda_a e_a.
\end{equation}
Its equations (22)--(23) correctly yield
\begin{equation}\label{eq:BJPS-valid}
g(he_a,e_a^*)=0,\qquad
h^2e_a=(1-\lambda_a^2)e_a,\qquad
h^2e_a^*=(1-\lambda_a^2)e_a^*.
\end{equation}
Write $H_{jk}=g(h\varepsilon_j,\varepsilon_k)$. The argument next chooses an orthogonal matrix $A_0$ for which $\overline H=A_0HA_0^{-1}$ is a real skew-symmetric normal form. Let $\alpha$ be the corresponding orthogonal transformation and put $\overline\varepsilon_j=\alpha\varepsilon_j$.

Equations (31)--(32) of \cite{BJPS} use both
\begin{equation}\label{eq:BJPS-two}
\overline H_{jk}=g(\overline h\,\overline\varepsilon_j,
\overline\varepsilon_k)
\qquad\text{and}\qquad
\overline h=\alpha h\alpha^{-1}.
\end{equation}
However, these two statements imply
\begin{align*}
\overline H_{jk}=g(\alpha h\alpha^{-1}(\alpha\varepsilon_j),
\alpha\varepsilon_k)=g(\alpha h\varepsilon_j,\alpha\varepsilon_k)
=g(h\varepsilon_j,\varepsilon_k)=H_{jk}.
\end{align*}
Thus \eqref{eq:BJPS-two} forces $\overline H=H$; it does not represent a new normal form unless the original matrix already had that form.
One may conjugate an endomorphism while keeping the basis fixed, or change the basis while keeping the endomorphism fixed. Conjugating the endomorphism and transporting the basis by the same orthogonal transformation leaves its matrix unchanged. This observation is independent of whether matrices are written using row or column conventions.

Consequently, the normal form for $h$ cannot be transferred back to the previously chosen $2$-planes in \eqref{eq:BJPS-frame} in the manner used to obtain equations (33)--(34) of \cite{BJPS}. Those equations would assert that $h$ preserves each such $2$-plane. This does not follow from a normal form theorem for one skew-adjoint endomorphism.


Our examples show explicitly what is not captured by the argument in
\cite{BJPS}: although $h$ is skew-adjoint, it need not preserve the
$2$-planes associated with the normal form of $\nabla\xi$.
In the family of Theorem~\ref{thm:main},
\[
\nabla\xi=-\lambda\varphi_3,\qquad h=\mu\varphi_2,
\qquad \lambda^2=1+\mu^2.
\]
Choose a horizontal unit vector $e_1$ and put
\[
e_1^*=\varphi_3e_1,\qquad
 e_2=\varphi_1e_1,\qquad e_2^*=\varphi_3e_2=\varphi_2e_1.
\]
These four vectors are orthonormal. The quaternionic relations give
\begin{equation}\label{eq:block-action}
he_1=\mu e_2^*,\qquad
he_1^*=\mu e_2,\qquad
he_2=-\mu e_1^*,\qquad
he_2^*=-\mu e_1.
\end{equation}
In the convention $H_{jk}=g(h\varepsilon_j,\varepsilon_k)$, the matrix is
\begin{equation}\label{eq:explicit-block}
H=\mu
\begin{pmatrix}
0&0&0&1\\
0&0&1&0\\
0&-1&0&0\\
-1&0&0&0
\end{pmatrix}.
\end{equation}
It has square $-\mu^2\Id$, but each coefficient $g(he_a,e_a^*)$ is zero. Rather than preserving the $2$-planes
$\Span\{e_1,e_1^*\}$ and $\Span\{e_2,e_2^*\}$, the operator $h$ interchanges them.
Here $\lambda_1=\lambda_2=\lambda$, so \eqref{eq:BJPS-valid} holds exactly:
\[
h^2=(1-\lambda^2)\Id_{\D}=-\mu^2\Id_{\D}.
\]
There is no contradiction with skew-adjointness. A real normal form for $h$ instead uses the pairs $(e_1,e_2^*)$ and $(e_1^*,e_2)$, which are not the pairs in \eqref{eq:BJPS-frame}. For $\mu>0$, this is precisely the quaternionic block described in Proposition~\ref{prop:quaternionic-blocks}.

\end{document}